\documentclass{amsart}
\usepackage{graphicx,amsmath,amsfonts,amssymb,amsthm,mathtools,fullpage,url} 
\usepackage[dvipsnames]{xcolor}

\theoremstyle{plain}
\newtheorem{theorem}{Theorem}[section]
\newtheorem{corollary}[theorem]{Corollary}
\newtheorem{lemma}[theorem]{Lemma}
\newtheorem{proposition}[theorem]{Proposition}
\theoremstyle{definition}

\newtheorem*{question*}{Question}

\DeclareMathOperator{\Trace}{trd}
\DeclareMathOperator{\Norm}{nrd}

\def\Spec{\operatorname{Spec}}
\def\End{\operatorname{End}}

\renewcommand{\phi}{\varphi}

\title{Inseparable endomorphisms and rank-2 sublattices of the Gross lattice}

\author[Aubry]{Yves Aubry}
\address[Aubry]{Institut de Math\'ematiques de Toulon - IMATH, Universit\'e de Toulon, France}
\email{yves.aubry@univ-tln.fr}

\address[Aubry]{Institut de Math\'ematiques de Marseille - I2M, Aix Marseille Univ, UMR 7373 CNRS, France}
\email{yves.aubry@univ-amu.fr}

\author[Vincent]{Christelle Vincent}
\address[Vincent]{University of Vermont, Burlington, United States of America}
\email{christelle.vincent@uvm.edu}

\begin{document}

\begin{abstract}
We answer a question posed by Love asking about a correspondence between isogenies from a supersingular elliptic curve to its Frobenius base-change and rank-2 sublattices of its Gross lattice.
We recast the question as one about the inseparable endomorphisms of the curve, and show that the correspondence holds when the trace of the endomorphism is zero, and may not hold otherwise.
\end{abstract}



\thanks{The authors would like to thank Jonathan Love for fruitful discussions which improved the article. In addition they are grateful for the hospitality of the GAATI laboratory at the Universit\'{e} de Polyn\'{e}sie fran\c{c}aise, where they completed this work.
Aubry is partially supported by the French Agence Nationale de la Recherche through the Barracuda project under Contract ANR-21-CE39-0009-BARRACUDA.
Vincent is supported by a Simons Foundation Travel Support for Mathematicians grant.}

\date{\today}

\maketitle

\section{Introduction}

Let $p$ be a prime, $\mathbb{F}_p$ be the finite field with $p$ elements, and $\overline{\mathbb{F}}_p$ be an algebraic closure of $\mathbb{F}_p$. 
Further, let $E$ be a supersingular elliptic curve defined over $\overline{\mathbb{F}}_p$ and $B_{p,\infty}$ be the quaternion algebra over $\mathbb{Q}$ ramified at $p$ and $\infty$.
Then the geometric endomorphism ring of $E$, $\End(E)$, is isomorphic to a maximal ideal $\mathcal{O}$ of $B_{p,\infty}$. 
For $\varphi(z)=z^p$ the Frobenius homomorphism of $\mathbb{F}_p$, we write
$$E^{(p)}:=E\otimes_{\Spec(\mathbb{F}_p), \varphi}\Spec(\mathbb{F}_p)$$
for the Frobenius base-change of $E$.

Let $\overline{\cdot}$ denote conjugation on $B_{p,\infty}$, then we write $\Trace(x) = x + \overline{x}$ for the reduced trace on $B_{p,\infty}$, and $\Norm(x) = x \overline{x}$ for the reduced norm.
These operations give rise to an inner product $\frac{1}{2}\Trace(x\overline{y})$ defined on $B_{p,\infty}$; we note that $\frac{1}{2}\Trace(x\overline{x})=\Norm(x)$.
For $\Lambda$ a lattice in $B_{p,\infty}$ with basis $\{x_1, \ldots, x_n \}$, we define its \textbf{determinant} to be
\begin{equation*}
\det \Lambda := \det \left( \frac{1}{2}\Trace(x_i\overline{x}_j) \right)_{1 \leq i,j \leq n};
\end{equation*}
this quantity is independent of the basis chosen to compute it.
In addition, for a lattice $\Lambda$ in $B_{p,\infty}$, we define its \textbf{Gross lattice} to be 
\begin{equation*}
\Lambda^T := \{ 2x - \Trace(x) : x \in \Lambda \}.
\end{equation*}
When $\Lambda = \End(E)$ for $E$ a supersingular elliptic curve defined over $\overline{\mathbb{F}}_p$, then we call $\End(E)^T$ the \textbf{Gross lattice of the elliptic curve $E$}.

In \cite{lovetalk}, Love asks the following:

\begin{question*}
Let $\ell$ be a positive integer, $E$ be a supersingular elliptic curve defined over $\overline{\mathbb{F}}_p$, $E^{(p)}$ be its Frobenius base-change, and $\End(E)$ be the geometric endomorphism ring of $E$. Does there exist an isogeny $E \to E^{(p)}$ of degree $\ell$ if and only if $\End(E)^T$ contains a rank-2 sublattice of determinant $4\ell p$?
\end{question*}

In this article, we recast this question as one about the inseparable endomorphisms of the supersingular elliptic curve $E$ (see Section \ref{preliminaries}) and we answer this new question partly in the positive, and partly in the negative. More precisely, our main result is:
\begin{theorem}\label{maintheorem}
Let $p$ be a prime and $E$ be a supersingular elliptic curve defined over $\overline{\mathbb{F}}_p$. Then for any positive integer $\ell$, there exists an endomorphism of $E$ of degree $\ell p$ and trace zero if and only if $\End(E)^T$ contains a rank-2 sublattice of determinant $4\ell p$.
\end{theorem}

This main result is a generalization of a result implicit in \cite[pages 851-852]{kaneko1989}, which states that if a supersingular elliptic curve has $j$-invariant belonging to $\mathbb{F}_p$ (and hence an endomorphism of degree $p$ and trace zero) then its Gross lattice $\End(E)^T$ contains a sublattice of rank $2$ and determinant $4p$, and of \cite[Proposition 3.1.1]{HKTV} which proves the converse to Kaneko's result.

In addition, in Section \ref{counterexample} we give examples of supersingular elliptic curves over $\overline{\mathbb{F}}_{p}$ which admit an endomorphism of trace $p$ and of degree $\ell p$, but we show that their Gross lattices do not contain a rank-2 sublattice of determinant $4\ell p$. To show that this happens quite often, we give an example with $\ell = p$, one with $\ell \neq p$ where $\ell$ is prime, and finally one with $\ell$ composite.

Finally, since it may be of independent interest, we highlight here Proposition \ref{prop:mainprop}, which gives an explicit $\mathbb{Z}$-basis for the commutator ideal $[\mathcal{O}:\mathcal{O}]$ of an order $\mathcal{O}$ in $B_{p,\infty}$ in terms of a basis of the order $\mathcal{O}$.

\section{Preliminaries}

\subsection{Background} \label{preliminaries}

Recall that throughout $E$ is a supersingular elliptic curve defined over $\overline{\mathbb{F}}_p$.
We begin by recasting Love's question on isogenies $E \to E^{(p)}$ into a question about endomorphisms of $E$.
Let $\ell$ be a positive integer and $\alpha \in \End(E)$ be of degree $\ell p$.
As remarked in Subsection 2.3.5 of \cite{Goren-Love-long}, a supersingular elliptic curve defined over $\overline{\mathbb{F}}_p$ has a unique subgroup scheme of rank $p$, and it is the kernel of the Frobenius isogeny $E \to E^{(p)}$, which is of degree $p$.
Therefore, the endomorphism $\alpha$ factors into the Frobenius isogeny followed by a degree-$\ell$ isogeny $E^{(p)}\to E$. 
The dual of this isogeny then gives rise to an isogeny $E\to E^{(p)}$ of degree $\ell$.
Conversely, if we have an isogeny $E\to E^{(p)}$ of degree $\ell$ for $\ell$ a positive integer, then by composing it with the Verschiebung, which is the dual of the Frobenius isogeny, we obtain an endomorphism of $E$ of degree $\ell p$.
By this argument, it is clear that finding an isogeny $E\to E^{(p)}$ of degree $\ell$ is equivalent to finding an endomorphism of $E$ of degree $\ell p$. 
We also find, since the Verschiebung isogeny is inseparable when $E$ is supersingular, that the inseparable endomorphisms of a supersingular elliptic curve are exactly those of degree divisible by $p$.
 
Let $\mathcal{O}$ be a maximal order in $B_{p,\infty}$; this will be isomorphic to the geometric endomorphism ring $\End(E)$ of a supersingular elliptic curve $E$ defined over $\overline{\mathbb{F}}_p$.
We note that under such an isomorphism, the degree of an endomorphism corresponds to the reduced norm of its image, and the trace on $\End(E)$ agrees with the reduced trace on $\mathcal{O}$.
The ideal of $\mathcal{O}$ containing the elements of reduced norm divisible by $p$ is a two-sided prime ideal of $\mathcal{O}$, which we denote $P$.
This is in fact the unique maximal two-sided ideal of $\mathcal{O}$ of reduced norm $p$.
Furthermore, let $[\mathcal{O},\mathcal{O}]$ be the two-sided \textbf{commutator ideal} of $\mathcal{O}$, by which we mean the two-sided ideal of $\mathcal{O}$ generated by elements of the form $[\alpha_1,\alpha_2] := \alpha_1\alpha_2 - \alpha_2 \alpha_1$ for $\alpha_i \in \mathcal{O}$.
It is well-known (see for example \cite[Section 42.4.6]{JVbook}) that $[\mathcal{O},\mathcal{O}]$ is equal to $P$ for any maximal order of $B_{p,\infty}$.
Finally, $P = [\mathcal{O},\mathcal{O}]$ is free of rank $4$, since it is a submodule of the free $\mathbb{Z}$-module $\mathcal{O}$ and of finite index $p^2$.
This last fact is the case because $\mathcal{O}/P \cong \mathbb{F}_{p^2}$, again as argued in \cite[Section 42.4.6]{JVbook}.

Define the linear map $\tau \colon B_{p,\infty} \to B_{p,\infty}$ by $\tau(x) = 2x - \Trace(x)$; then the image of a lattice $\Lambda$ under $\tau$ is exactly its Gross lattice $\Lambda^T$. As remarked in \cite[Section 3.3]{Goren-Love-short}, if we restrict $\tau$ to an order of $B_{p,\infty}$, then $\ker(\tau) = \mathbb{Z}$. In our proof we will need the following result:
\begin{lemma}\label{petitlemme}
Let $\mathcal{O}$ be an order in $B_{p,\infty}$. If $\{1,\alpha_1,\alpha_2,\alpha_3\}$ is a basis for $\mathcal{O}$, then $\{\tau(\alpha_1), \tau(\alpha_2), \tau(\alpha_3)\}$ is a basis for $\mathcal{O}^T$. Conversely, if $\{\beta_1,\beta_2,\beta_3\}$ is a basis for $\mathcal{O}^T$ then $\{1, \alpha_1,\alpha_2,\alpha_3\}$ is a basis for $\mathcal{O}$ for any $\alpha_i$ such that $\alpha_i= \tau(\beta_i)$ for $i = 1,2,3$.
\end{lemma}

\begin{proof}
Begin by supposing that $\{1,\alpha_1,\alpha_2,\alpha_3\}$ is a basis for $\mathcal{O}$, and let $\beta \in \mathcal{O}^T$. Then $\beta = \tau(\alpha)$ for some $\alpha \in \mathcal{O}$, and there are integers $a_0, a_1,a_2,a_3$ such that $\alpha  = a_0 + a_1\alpha_1+ a_2 \alpha_2 + a_3 \alpha_3$. Hence since $\tau$ is $\mathbb{Z}$-linear and $\tau(a_0) = 0$, we have that $\beta = a_1 \tau(\alpha_1) + a_2 \tau(\alpha_2) + a_3 \tau(\alpha_3)$, and the set $\{\tau(\alpha_1), \tau(\alpha_2), \tau(\alpha_3)\}$ generates $\mathcal{O}^T$. In addition if $a_1 \tau(\alpha_1) + a_2 \tau(\alpha_2) + a_3 \tau(\alpha_3) = 0$ for some integers $a_1,a_2,a_3$, then $a_1\alpha_1+ a_2 \alpha_2 + a_3 \alpha_3 \in \ker(\tau)=\mathbb{Z}$. If we denote by $a_0$ this integer, then $a_0 -a_1\alpha_1- a_2 \alpha_2 - a_3 \alpha_3 =0$. Since the set $\{1,\alpha_1,\alpha_2,\alpha_3\}$ is linearly independent, we conclude that $a_0 = a_1 = a_2 = a_3 = 0 $ and therefore the set $\{\tau(\alpha_1), \tau(\alpha_2), \tau(\alpha_3)\}$ is linearly independent.

Now let $\{\beta_1,\beta_2,\beta_3\}$ be a basis for $\mathcal{O}^T$, $\alpha_i \in \mathcal{O}$ such that $\tau(\alpha_i) = \beta_i$ for each $i$, and $\alpha \in \mathcal{O}$. We have that there are integers $a_1, a_2, a_3$ such that $\tau(\alpha) = a_1 \beta_1 + a_2 \beta_2 + a_3 \beta_3$ since $\tau(\alpha) \in \mathcal{O}^T$. Because $\tau(\alpha_i) = \beta_i$ and $\tau$ is $\mathbb{Z}$-linear, we thus have $\tau(\alpha) = \tau(a_1\alpha_1 + a_2 \alpha_2 + a_3 \alpha_3)$ and therefore $\alpha - a_1\alpha_1 - a_2 \alpha_2 - a_3 \alpha_3 \in \ker(\tau) = \mathbb{Z}$. Again denoting by $a_0$ this integer, we conclude that $\alpha = a_0 +a_1\alpha_1 + a_2 \alpha_2 + a_3 \alpha_3$ and the set $\{1,\alpha_1,\alpha_2,\alpha_3\}$ generates $\mathcal{O}$. In addition if $a_0 +a_1\alpha_1 + a_2 \alpha_2 + a_3 \alpha_3 = 0$, applying $\tau$ to this equation we obtain that $a_1 \beta_1 + a_2 \beta_2 + a_3 \beta_3 = 0$. Since the set $\{\beta_1,\beta_2,\beta_3\}$ is linearly independent, we have $a_1 = a_2 = a_3 = 0$. From this we deduce that $a_0 = 0$ as well, and therefore the set $\{1,\alpha_1,\alpha_2,\alpha_3\}$ is linearly independent.
\end{proof}

\subsection{Existence of an inseparable endomorphism}

We begin by proving one direction of our main result, Theorem \ref{maintheorem}. We note that a similar argument is given in the proof of Lemma 3.4 of \cite{Chevyrev-Galbraith}.

\begin{proposition}\label{half-love}
Let $E$ be a supersingular elliptic curve defined over $\overline{\mathbb{F}}_p$ and $\End(E)$ be its geometric endomorphism ring. If there exists a sublattice of rank 2 of $\End(E)^T$ of determinant $4\ell p$ then there exists an inseparable endomorphism of $E$ of degree $\ell p$ and trace zero.
\end{proposition}

\begin{proof}
Let $\Lambda$ be a sublattice of rank 2 of $\End(E)^T$ of determinant $4\ell p$, and $\{\gamma_1, \gamma_2\}$ be any basis of $\Lambda$. Consider the element
 \begin{equation*}
 \alpha:=\frac{1}{2}\gamma_1\overline{\gamma}_2-\frac{1}{4} \Trace (\gamma_1\overline{\gamma}_2) \in B_{p,\infty}.
 \end{equation*}

We show, as in the proof of Proposition 3.1 of \cite{HKTV}, that $\alpha$ belongs to $\End(E)$ and has degree $\ell p$ and trace zero. Indeed, consider any elements $\alpha_1$ and $\alpha_2$ in $\End(E)$ such that $\gamma_1=\tau(\alpha_1)$
and $\gamma_2=\tau(\alpha_2)$. Then we have
$$\alpha=2\alpha_1\overline{\alpha}_2 -(\Trace(\alpha_2)\alpha_1+\Trace(\alpha_1)\overline{\alpha}_2)
+\Trace(\alpha_1)\Trace(\alpha_2)
-\Trace(\alpha_1\overline{\alpha}_2)$$
which implies that $\alpha\in{\mathcal O}$.

Moreover, we have:
$$\Norm(\alpha)
= \frac{1}{4} \left(\Norm(\gamma_1)\Norm(\gamma_2)-\frac{1}{4}\Trace(\gamma_1\overline{\gamma}_2)^2\right)
=\frac{1}{4}\det \Lambda=\ell p,$$

as well as 
$$\Trace(\alpha)=\frac{1}{2}\gamma_1\overline{\gamma}_2-\frac{1}{4} \Trace(\gamma_1\overline{\gamma}_2)
+ \frac{1}{2}\gamma_2\overline{\gamma}_1-\frac{1}{4} \Trace(\gamma_1\overline{\gamma}_2)
= \frac{1}{2} \Trace(\gamma_1\overline{\gamma}_2)-\frac{1}{2} \Trace(\gamma_1\overline{\gamma}_2)=0.$$

Finally we can conclude that $\alpha$ is inseparable since its degree is divisible by $p$ and $E$ is supersingular.
\end{proof}

\section{Existence of a sublattice of rank $2$}

We now turn our attention to the converse of Proposition \ref{half-love}, which specifically asks to, given an endomorphism $\alpha$ of a supersingular curve $E$ of degree $\ell p$ and trace $0$, exhibit a rank-2 sublattice of $\End(E)^T$ of determinant $4\ell p$. We do this by proving that there are $\gamma_1$ and $\gamma_2$ linearly independent in $\End(E)^T$ such that 
\begin{equation*}
\alpha=\frac{1}{2}\gamma_1\overline{\gamma}_2-\frac{1}{4} \Trace(\gamma_1\overline{\gamma}_2).
\end{equation*}

Indeed, if this is the case, by the proof of Proposition \ref{half-love}, we obtain that if $\Lambda$ is the lattice generated by the set $\{ \gamma_1,\gamma_2\}$, then $\det \Lambda = 4\Norm(\alpha)$, and hence the determinant of the rank-2 sublattice of $\End(E)^T$ spanned by $\gamma_1$ and $\gamma_2$ is $4\ell p$ (which also proves that $\gamma_1$ and $\gamma_2$ are linearly independent since $\ell \neq 0$).

To do this, we begin by exhibiting an explicit $\mathbb{Z}$-basis for $[\mathcal{O},\mathcal{O}]$.

\begin{proposition}\label{prop:mainprop}
Let $\mathcal{O}$ be a maximal order of $B_{p,\infty}$, $[\mathcal{O},\mathcal{O}]$ its commutator ideal, and $\{1, \alpha_1,\alpha_2,\alpha_3\}$ a basis for $\mathcal{O}$. Then the set 
\begin{equation*}
\left\{[\alpha_1,\alpha_2],[\alpha_1,\alpha_3], [\alpha_2,\alpha_3], \alpha_1 [\alpha_2,\alpha_3]\right\}
\end{equation*}
is a $\mathbb{Z}$-basis of $[\mathcal{O},\mathcal{O}]$.
\end{proposition}

\begin{proof}
Let $C$ be the $\mathbb{Z}$-lattice generated by the set $\{[\alpha_1,\alpha_2],[\alpha_1,\alpha_3], [\alpha_2,\alpha_3]\}$. We first claim that $C$ contains $[\gamma_1,\gamma_2]$, the commutator of $\gamma_1$ and $\gamma_2$, for all $\gamma_1,\gamma_2$ in $\mathcal{O}$. This follows by writing each $\gamma_i$ in the basis $\{1,\alpha_1,\alpha_2,\alpha_3\}$ and computing the commutator explicitly.

By the definition of the commutator ideal and two-sided ideals in noncommutative rings, the elements of $[\mathcal{O},\mathcal{O}]$ are of the form
\begin{equation*}
l_1 [x_1,y_1] r_1 + l_2 [x_2,y_2] r_2 + \cdots + l_n [x_n,y_n] r_n  
\end{equation*}
for $n$ any positive integer and $l_i,r_i,x_i$ and $y_i$ elements of $\mathcal{O}$ for $i \in \{1, \ldots, n\}$. Hence elements of $[\mathcal{O},\mathcal{O}]$ are of the form
\begin{equation}\label{eq:bigsum}
l_1 [\alpha_{1,i},\alpha_{1,j}]r_1 + l_2 [\alpha_{2,i},\alpha_{2,j}]r_2 + \cdots + l_n [\alpha_{n,i},\alpha_{n,j}]r_n  
\end{equation}
where as before $l_i$ and $r_i$ are elements of $\mathcal{O}$ and each $\alpha_{i,j} \in \{\alpha_1,\alpha_2,\alpha_3\}$.

Writing each $l_i$ and $r_i$ in the basis $\{1,\alpha_1,\alpha_2,\alpha_3\}$, we see that to show that \eqref{eq:bigsum} is a $\mathbb{Z}$-linear combination of elements of the set $\left\{[\alpha_1,\alpha_2],[\alpha_1,\alpha_3], [\alpha_2,\alpha_3], \alpha_1 [\alpha_2,\alpha_3]\right\}$, it now suffices to show that elements of the form $\alpha_i[\alpha_j,\alpha_k]$, $[\alpha_i,\alpha_j]\alpha_k$, and $\alpha_i[\alpha_j,\alpha_k]\alpha_{\ell}$ for any $i,j,k,\ell$ are all $\mathbb{Z}$-linear combinations of elements of the set $\left\{[\alpha_1,\alpha_2],[\alpha_1,\alpha_3], [\alpha_2,\alpha_3], \alpha_1 [\alpha_2,\alpha_3]\right\}$. We do this in several steps.

First, consider products of the form $\alpha_i[\alpha_j,\alpha_k]$. If $i =j$ or $i = k$, we obtain
\begin{gather*}
\alpha_i[\alpha_i,\alpha_j] =-\alpha_i[\alpha_j,\alpha_i]=  [\alpha_i,\alpha_i \alpha_j] \in C.
\end{gather*}
If $i,j,$ and $k$ are all distinct, we compute
\begin{equation*}
\alpha_i[\alpha_j,\alpha_k] + \alpha_j[\alpha_i,\alpha_k] = [\alpha_i, \alpha_i\alpha_k] + [\alpha_i, \alpha_j \alpha_k] \in C
\end{equation*}
and
\begin{equation*}
\alpha_i[\alpha_j,\alpha_k] + \alpha_k[\alpha_j,\alpha_i] = [\alpha_i\alpha_j,\alpha_k] + [\alpha_k \alpha_j, \alpha_i] \in C.
\end{equation*}
Thus the left ideal generated by the commutators in $\mathcal{O}$ has $\mathbb{Z}$-basis $\left\{[\alpha_1,\alpha_2],[\alpha_1,\alpha_3], [\alpha_2,\alpha_3], \alpha_1 [\alpha_2,\alpha_3]\right\}$.

Therefore to prove our claim, it now suffices to show that $[\alpha_i,\alpha_j]\alpha_k$ for any $i,j,k$ and $\alpha_1[\alpha_2,\alpha_3]\alpha_i$ for any $i$ are in the lattice generated by the elements $\left\{[\alpha_1,\alpha_2],[\alpha_1,\alpha_3], [\alpha_2,\alpha_3], \alpha_1 [\alpha_2,\alpha_3]\right\}$.

Again, if $i = j$ or $i = k$, we have
\begin{gather*}
[\alpha_i,\alpha_j] \alpha_i =-[\alpha_j,\alpha_i] \alpha_i=  [\alpha_i,\alpha_j \alpha_i] \in C,
\end{gather*}
and if $i,j,k$ are all distinct,
\begin{equation*}
[\alpha_j,\alpha_k]\alpha_i + [\alpha_i,\alpha_k]\alpha_j = [\alpha_i, \alpha_k\alpha_j] + [\alpha_j, \alpha_k \alpha_i] \in C
\end{equation*}
and
\begin{equation*}
[\alpha_j,\alpha_k]\alpha_i + [\alpha_j,\alpha_i]\alpha_k = [\alpha_j\alpha_k,\alpha_i] + [\alpha_j \alpha_i, \alpha_k] \in C.
\end{equation*}
In addition, we have, of course,
\begin{equation*}
\alpha_1[\alpha_2,\alpha_3] - [\alpha_2,\alpha_3] \alpha_1 = [\alpha_1, [\alpha_2,\alpha_3]] \in C,
\end{equation*}
from which it follows that each $[\alpha_i,\alpha_j]\alpha_k$, for any $i,j,k$, belongs to the lattice generated by the elements $\left\{[\alpha_1,\alpha_2],[\alpha_1,\alpha_3], [\alpha_2,\alpha_3], \alpha_1 [\alpha_2,\alpha_3]\right\}$.

Finally, considering elements of the form $\alpha_1[\alpha_2,\alpha_3]\alpha_i$, we note that by the proof above the product $[\alpha_2,\alpha_3]\alpha_i$ for each $i$ belongs to the lattice generated by the elements $\left\{[\alpha_1,\alpha_2],[\alpha_1,\alpha_3], [\alpha_2,\alpha_3], \alpha_1 [\alpha_2,\alpha_3]\right\}$, which is the left ideal generated by the commutators of $\mathcal{O}$. Hence $\alpha_1 [\alpha_2,\alpha_3]\alpha_i$ also belongs to this left ideal and therefore to our lattice. Therefore the set $\left\{[\alpha_1,\alpha_2],[\alpha_1,\alpha_3], [\alpha_2,\alpha_3], \alpha_1 [\alpha_2,\alpha_3]\right\}$ generates $[\mathcal{O},\mathcal{O}]$.

Now since $[\mathcal{O},\mathcal{O}]$ is free of rank $4$, a set of four elements that span it must be linearly independent.
\end{proof}

Of interest for us will be the following corollary:

\begin{corollary}\label{cor:betabasis}
Let $\mathcal{O}$ be a maximal order of $B_{p,\infty}$, $[\mathcal{O},\mathcal{O}]$ its commutator ideal, and $\{\beta_1,\beta_2,\beta_3\}$ be a basis for $\mathcal{O}^T$. Then the set 
\begin{equation}
\left\{\frac{\beta_1\beta_2 - \beta_2\beta_1}{4}, \frac{\beta_1 \beta_3 - \beta_3\beta_1}{4}, \frac{\beta_2\beta_3 - \beta_3\beta_2}{4}\right\}
\end{equation}
is a $\mathbb{Z}$-basis for the sublattice of elements of trace zero in $[\mathcal{O},\mathcal{O}]$ .
\end{corollary}

\begin{proof}
For each $i = 1,2,3$, let $\alpha_i \in \mathcal{O}$ be such that $\tau(\alpha_i) = \beta_i$. Since $\{\beta_1, \beta_2,\beta_3\}$ is a basis for $\mathcal{O}^T$, by Lemma~\ref{petitlemme}, the set $\{1,\alpha_1,\alpha_2,\alpha_3\}$ is a basis for $\mathcal{O}$.
A straightforward computation shows that 
\begin{equation}
[\alpha_i,\alpha_j] = \alpha_i\alpha_j - \alpha_j \alpha_i =\frac{ \beta_i\beta_j-\beta_j\beta_i}{4}
\end{equation}
for each $i,j \in \{1,2,3\}$.

The commutator ideal $[\mathcal{O},\mathcal{O}]$ contains elements of nonzero trace; this follows from the fact that $[\mathcal{O}, \mathcal{O}]  = P \supset p \mathcal{O}$. Since each $[\alpha_i,\alpha_j]$ has trace zero, the element $\alpha_1[\alpha_2,\alpha_3]$ must have nonzero trace. It follows then that if $\alpha \in [\mathcal{O}, \mathcal{O}]$ is of trace zero with
\begin{equation*}
\alpha  = a_0 [\alpha_1,\alpha_2] + a_1 [\alpha_1,\alpha_3] + a_2 [\alpha_2,\alpha_3] + a_3 \alpha_1[\alpha_2,\alpha_3]
\end{equation*}
we must have $a_3 = 0$. Hence the set $\{[\alpha_1, \alpha_2],[\alpha_1,\alpha_3],[\alpha_2,\alpha_3]\}$ spans the sublattice of elements of trace zero in $[\mathcal{O},\mathcal{O}]$, and since they are elements of a basis, they are linearly independent.
\end{proof}

We are now ready to prove the following result, which along with Proposition~\ref{half-love} completes the proof of Theorem~\ref{maintheorem}:

\begin{theorem}
Let $E$ be a supersingular elliptic curve defined over $\overline{\mathbb{F}}_p$. If $\alpha$ is an endomorphism of $E$ of degree $\ell p$ and trace $0$, then there are $\gamma_1$ and $\gamma_2$ in $\End(E)^T$ spanning a lattice of determinant $4\ell p$.
\end{theorem}

\begin{proof}
Let $\mathcal{O}$ be a maximal order of $B_{p,\infty}$ which is isomorphic to $\End(E)$. Since $\alpha$ has norm divisible by $p$ and trace $0$, we have that $\alpha$ belongs to the $\mathbb{Z}$-lattice of elements of trace zero in $[\mathcal{O},\mathcal{O}]$. Hence by Corollary \ref{cor:betabasis}, there are integers $a_1, a_2, a_3$ such that
\begin{equation*}
\alpha = a_1 \frac{\beta_1\beta_2-\beta_2\beta_1}{4} +a_2 \frac{\beta_1\beta_3-\beta_3\beta_1}{4} + a_3 \frac{\beta_2\beta_3-\beta_3\beta_2}{4}. 
\end{equation*}

We claim that there are integers $c_{11},c_{12},c_{13},c_{21},c_{22},c_{23}$ such that
\begin{equation*}
\left\{
\begin{matrix}
a_1 = c_{12} c_{21} - c_{11}c_{22}   \\
a_2 = c_{13}c_{21}- c_{11}c_{23} \\
a_3 = c_{13} c_{22} - c_{12}c_{23}. \cr
\end{matrix}
\right. 
\end{equation*}
Indeed if $a_1 = a_2 = 0$, we have the integer solution $c_{11} = c_{13} = c_{21} = c_{22} = 0$, $c_{12} = a_3$ and $c_{23} = -1$. Otherwise we have the solution $c_{11} = \gcd(a_1,a_2)$, $c_{21} = 0$, $c_{22} = - \frac{a_1}{c_{11}}$, $c_{23} = - \frac{a_2}{c_{11}}$ and $c_{12}, c_{13}$ solutions of the diophantine equation
\begin{equation*}
a_2 c_{12} - a_1 c_{13} = c_{11} a_3,
\end{equation*}
which exist since $\gcd(a_1,a_2)$ divides $c_{11} a_3$.

Now let $\{ \beta_1, \beta_2, \beta_3 \}$ be a $\mathbb{Z}$-basis for $\End(E)^T$ and set $\gamma_1 = c_{11} \beta_1 + c_{12} \beta_2 + c_{13} \beta_3$ and $\gamma_2 = c_{21} \beta_1 + c_{22} \beta_2 + c_{23} \beta_3$, so that $\gamma_1,\gamma_2 \in \End(E)^T$. A tedious but elementary computation shows that 
\begin{equation*}
\alpha = \frac{1}{2} \gamma_1 \overline{\gamma}_2 - \frac{1}{4} \Trace( \gamma_1 \overline{\gamma}_2 ).
\end{equation*}
As remarked, by the proof of Proposition \ref{half-love}, the determinant of the lattice $\Lambda$ spanned by $\gamma_1$ and $\gamma_2$ is then $4 \Norm(\alpha) = 4 \ell p$. Since this is nonzero, the set $\{\gamma_1,\gamma_2\}$ is linearly independent and $\Lambda$ is of rank $2$.
\end{proof}

\section{Counter-examples when the trace is not zero}\label{counterexample}

One may rightly wonder if in Theorem \ref{maintheorem} the condition that the element $\alpha \in \End(E)$ have trace zero is necessary. We now give several examples that show that it is not guaranteed that there is a rank-2 sublattice of $\End(E)^T$ with determinant $4 \ell p$ if $\End(E)$ contains an element of degree $\ell p$ but whose trace is nonzero. 

\subsection{Counter-example when $\ell = p$}\label{subsec:lisp}

Consider the elliptic curve $E$ defined over $\mathbb{F}_{11}$ with $j$-invariant $j=0$; this curve is supersingular.
In this case, $\End(E)$ has an endomorphism of norm $11^2=121$ but $\End(E)^T$ does not admit a rank-2 sublattice of determinant $4\times 121$, as we now demonstrate.

Let $B_{11,\infty}$ be the quaternion algebra over $\mathbb Q$ ramified exactly at $11$ and $\infty$ and
let $i,j,k$ be elements of $B_{11,\infty}$ such that $i^2=-3$, $j^2=-11$ and $k=ij$.
Let $\mathcal O$ be any maximal order of $B_{11,\infty}$ such that $\End(E)\simeq \mathcal O$.

Then there is a choice of $\mathcal{O}$ such that its Gross lattice ${\mathcal O}^T$ has a successive minimal basis given by $\{i,\frac{i+3j-k}{3},\frac{-i-2k}{3}\}$
and the Gram matrix of this basis is given by (see Proposition 3.2.1 of \cite{HKTV}):
$$
\begin{pmatrix}
3 & 1 &  1 \\
1 & 15 & -7\\
1 & -7 & 15\cr
\end{pmatrix}.
$$
For any other choice $\mathcal{O}'$ of a maximal ideal of $B_{p,\infty}$ isomorphic to $\mathcal{O}$, the Gram matrix of a successive minimal basis remains the same.

Now consider the following element of $B_{p,\infty}$:
$$\alpha:=\frac{11}{2} + \frac{11}{2}i;$$
we have that $\alpha$ has reduced trace $\Trace(\alpha)=11$ and reduced norm $\Norm(\alpha)=11 \times 11$. There is a choice of maximal order $\mathcal{O}$ of $B_{11,\infty}$ such that $\End(E)\simeq \mathcal O$ and $\alpha \in \mathcal{O}$. We show that $\mathcal{O}^T$ does not contain a sublattice of rank 2 and determinant $4 \times 121$.

Indeed, let $\gamma_1$ and $\gamma_2$ be two elements of $\mathcal{O}^T$ and $\Lambda$ be the lattice that they generate. 
If as before we write $\{\beta_1,\beta_2,\beta_3\}$ for a basis for $\mathcal{O}^T$, then $\gamma_1 = c_{11} \beta_1 + c_{12} \beta_2 + c_{13} \beta_3$ and $\gamma_2 = c_{21} \beta_1 + c_{22} \beta_2 + c_{23} \beta_3$ for some integers $c_{ij}$. Then we have
\begin{equation*}
\Norm(\gamma_i) = 3 c_{i1}^2 + 2 c_{i1} c_{i2} + 2 c_{i1} c_{i3} +15 c_{i2}^2 -14 c_{i2} c_{i3} + 15 c_{i3}^2
\end{equation*}
as well as
\begin{equation*}
\Trace(\gamma_1 \overline{\gamma}_2) = 6 c_{11}c_{21} + 2 c_{11} c_{22} + 2 c_{11} c_{23} + 2 c_{12} c_{21} + 30 c_{12} c_{22} -14 c_{12} c_{23} +2 c_{13} c_{21} - 14 c_{13} c_{22} +30 c_{13} c_{23}.
\end{equation*}

The determinant of $\Lambda$ is then $\Norm(\gamma_1)\Norm(\gamma_2) - \frac{1}{4} \Trace(\gamma_1\overline{\gamma}_2)^2$, which is a homogenous polynomial of degree 4 with 6 variables.
After dividing by $44$ and performing the following change of variables:
\begin{equation}\label{eq:variables}
\left\{
\begin{matrix}
x=c_{11}c_{22}-c_{12}c_{21}\\
y=c_{11}c_{23}-c_{13}c_{21}\\
z= c_{12}c_{23}-c_{13}c_{22}\cr
\end{matrix}
\right.
\end{equation}
we obtain the quadratic form
$$
Q=x^2+y^2+4z^2-xy+yz-xz.
$$
It remains then to prove that this quadratic form does not represent $11$.
Indeed, $Q$ has the following diagonal form:
$$
Q=\left( x-\frac{y+z}{2}\right)^2 + \frac{3}{4}\left(y+\frac{z}{3}\right)^2 + \frac{11}{3}z^2,
$$
so that
$$12Q=3(2x-y-z)^2+(3y+z)^2+44z^2.$$

The equation $3(2x-y-z)^2+(3y+z)^2+44z^2=132$ implies immediately that $44z^2\leq 132$ which gives $z^2\leq 3$ and thus $z\in\{-1,0,1\}$.

If $z=0$ then we obtain that $(2x-y)^2+3y^2=44$. Then $3y^2 \leq 44$,
which gives $3y^2\in\{0,3,12,27\}$ and then
$(2x-y)^2 \in \{44, 41, 32,17\}$, but none of these integers is a perfect square.
Thus we have no integer solution in the case $z=0$.

If $z=\pm 1$, we proceed in the same way and we conclude again that there is no integer solution, and therefore no sublattice of rank $2$ of $\mathcal{O}^T$ with determinant $4 \times 121$.

\subsection{Counter-example when $\ell \neq p$ and $\ell$ is prime}

Consider the supersingular elliptic curve $E$ defined over $\mathbb{F}_{31}$ with $j$-invariant $j = 4$. Let $B_{31,\infty}$ be the quaternion algebra over $\mathbb{Q}$ ramified at $31$ and $\infty$, let $i,j,k$ be elements of $B_{31,\infty}$ such that $i^2=-1$, $j^2=-31$ and $k = ij$, and $\mathcal{O}$ be a maximal order of $B_{31,\infty}$ such that $\End(E)\simeq \mathcal O$ and such that $\mathcal{O}$ contains the element
\begin{equation*}
\alpha = \frac{31}{2} - \frac{31}{3} i - \frac{7}{6}j - \frac{2}{3} k.
\end{equation*}
We have that $\Trace(\alpha) = 31$ and $\Norm(\alpha) = 13 \times 31$.

Similarly to the counter-example above, we can compute the determinant form of a rank-$2$ sublattice of the Gross lattice $\mathcal{O}^T$. Using the notation of Subsection \ref{subsec:lisp} and the change of variables of Equation \eqref{eq:variables}, this form is equal to $124Q$ where $Q$ is 
\begin{equation*}
Q = x^2 + 2y^2 +5z^2-xy-xz+2yz.
\end{equation*}
We can see that the form $Q$ does not represent $13$ using its diagonal form
\begin{equation*}
Q = \left(x - \frac{y+z}{2} \right)^2 +\frac{7}{4} \left(y + \frac{3}{7}z\right)^2 + \frac{31}{7}z^2
\end{equation*}
in a manner entirely analogous to that employed in Subsection \ref{subsec:lisp}.

\subsection{Counter-example when $\ell \neq p$ and $\ell$ is composite}
 
This time let $E$ be the supersingular elliptic curve defined over $\mathbb{F}_{19}$ with $j$-invariant $j = 7$, let $B_{19,\infty}$ be the quaternion algebra over $\mathbb{Q}$ ramified at $19$ and $\infty$, let $i,j,k$ be elements of $B_{19,\infty}$ such that $i^2=-1$, $j^2=-19$ and $k = ij$. Furthermore, choose $\mathcal{O}$ a maximal order of $B_{19,\infty}$ such that $\End(E)\simeq \mathcal O$ and $\mathcal{O}$ contains the element
 \begin{equation*}
 \alpha = \frac{19}{2} - \frac{19}{2}i - \frac{1}{2}j - \frac{1}{2}k
 \end{equation*}
which is of reduced norm $10 \times 19$ and reduced trace $19$. 

Once more we can show that the Gross lattice $\mathcal{O}^T$ does not contain a sublattice of rank $2$ and determinant $4 \times 10 \times 19$. This time, the determinant form of a rank-$2$ sublattice of $\mathcal{O}^T$ is equal to $76Q$ where 
\begin{equation*}
Q = x^2 + 2y^2 + 3z^2 -xy - xz + yz
\end{equation*}
with diagonal form
\begin{equation*}
Q = \left(x-\frac{y+z}{2}\right)^2 + \frac{7}{4}\left(y + \frac{z}{7}\right)^2 + \frac{19}{7}z^2.
\end{equation*}
As this form does not represent $10$, there is no rank-$2$ sublattice of determinant $4 \times 10 \times 19$ in $\mathcal{O}^T$.

\bibliographystyle{alpha}
\bibliography{biblio}
\end{document}